\documentclass[12pt]{amsart}
\usepackage{amsfonts}
\usepackage{amsmath}
\usepackage{amssymb}
\usepackage{graphicx}
\usepackage{amscd}

\newtheorem{theorem}{Theorem}

\newtheorem{definition}[theorem]{Definition}

\newtheorem{lemma}[theorem]{Lemma}

\newtheorem{proposition}[theorem]{Proposition}
\newtheorem{remark}[theorem]{Remark}

\input{tcilatex}

\begin{document}
\author{A. Fotiadis}
\email{fotiadisanestis@math.auth.gr}
\author{N. Mandouvalos}
\email{nikosman@math.auth.gr}
\author{M. Marias}
\email{marias@math.auth.gr}
\curraddr{Department of Mathematics, Aristotle University of Thessaloniki,
Thessaloniki 54.124, Greece}
\date{}
\dedicatory{Dedicated to the memory of Georges Georganopoulos}
\title[Schr\"{o}dinger Equation]{Schr\"{o}dinger Equation on Locally
Symmetric Spaces}
\maketitle

\begin{abstract}
We prove dispersive and Strichartz estimates for Schr\"{o}-dinger equations
on a class of locally symmetric spaces $\Gamma \backslash X$, where $X=G/K$
is a symmetric space and $\Gamma $ is a torsion free discrete subgroup of $G$%
. We deal with the cases when either $X$ has rank one or $G$ is complex. We
present Strichartz estimates applications to the well-posedness and
scattering for nonlinear Schr\"{o}dinger equations.
\end{abstract}

\section{Introduction and statement of the results}

Let $M$ be a Riemannian manifold and denote by $\Delta $ its
Laplace-Beltrami operator. The nonlinear Schr\"{o}dinger equation (NLS) on $%
M $ 
\begin{equation}
\left\{ 
\begin{array}{l}
i\partial _{t}u\left( t,x\right) +\Delta _{x}u\left( t,x\right) =F\left(
u\left( t,x\right) \right) , \\ 
u\left( 0,x\right) =f\left( x\right) ,%
\end{array}%
\right.  \label{NLS}
\end{equation}
has been extensively studied the last thirty years. Its study relies on
precise estimates of the kernel $s_{t}$ of the Schr\"{o}dinger operator $%
e^{it\Delta }$, the heat kernel of pure imaginary time. The estimates of $%
s_{t}$ allow us to obtain \textit{dispersive estimates }of the operator $%
e^{it\Delta }$ of the form 
\begin{equation}
\left\Vert e^{it\Delta }\right\Vert _{L^{\widetilde{q}^{\prime }}\left(
M\right) \rightarrow L^{q}\left( M\right) }\leq c\psi \left( t\right) \text{%
, \ }t\in \mathbb{R},  \label{diest000}
\end{equation}%
for all $q,\widetilde{q}\in \left( 2,\infty \right] $, where $\psi $ is a
positive function and $\widetilde{q}^{\prime }$ is the conjugate of $%
\widetilde{q}$.

Dispersive estimates of $e^{it\Delta }$ as above, allow us to obtain \textit{%
Strichartz estimates }of the solutions $u\left( t,x\right) $ of (\ref{NLS}):%
\textit{\ } 
\begin{equation}
\left\Vert u\right\Vert _{L^{p}\left( \mathbb{R};L^{q}\left( M\right)
\right) }\leq c\left\{ \left\Vert f\right\Vert _{L^{2}\left( M\right)
}+\left\Vert F\right\Vert _{L^{\widetilde{p}^{\prime }}\left( \mathbb{R};L^{%
\widetilde{q}^{\prime }}\left( M\right) \right) }\right\} ,  \label{stest1}
\end{equation}%
for all pairs $\left( \frac{1}{p},\frac{1}{q}\right) $ and $\left( \frac{1}{%
\widetilde{p}},\frac{1}{\widetilde{q}}\right) $ which lie in a certain
triangle\textit{. }

Strichartz estimates have applications to well-posedness and scattering
theory for the NLS equation.

In the case of $\mathbb{R}^{n}$, the first such estimate was obtained by
Strichartz himself \cite{STRI} in a special case. Then, Ginibre and Velo 
\cite{GIVE} obtained the complete range of estimates except the case of
endpoints which were proved by Keel and Tao \cite{KETA}.

In view of the important applications to nonlinear problems, many attempts
have been made to study the dispersive properties for the corresponding
equations on various Riemannian manifolds (see e.g., \cite%
{ANPI,ANPIVA,BA,BACASTA,BOU,BU,HAS,IONSTA,PI1,PI2} and references within).
More precisely, dispersive and Strichartz estimates for the Schr\"{o}dinger
equation on real hyperbolic spaces have been stated by Banica \cite{BA},
Pierfelice \cite{PI1,PI2}, Banica et al. \cite{BACASTA}, Anker and
Pierfelice \cite{ANPI}, Ionescu and Staffilani \cite{IONSTA}. In a recent
paper Anker, Pierfelice and Vallarino \cite{ANPIVA} treat NLS in the context
of Damek-Ricci spaces, which include all rank one symmetric spaces of
noncompact type.

In the present work we treat NLS equations on a class of locally symmetric
spaces.

\subsection{The class $(S)$ of locally symmetric spaces}

In this section we describe the class of locally symmetric spaces on which
we shall treat NLS equations.

For the statement of the results we need to introduce some notation. For
more details see Section 2. Let $G$ be a semisimple Lie group, connected,
noncompact, with finite center and $K$ be a maximal compact subgroup of $G$.
We denote by $X$ the Riemannian symmetric space $G/K$.

Denote by $\mathfrak{g}$ and $\mathfrak{k}$ the Lie algebras of $G$ and $K$.
Let also $\mathfrak{p}$ be the subspace of $\mathfrak{g}$ which is
orthogonal to $\mathfrak{k}$ with respect to the Killing form. The Killing
form induces a $K$-invariant scalar product on $\mathfrak{p}$ and hence a $G$%
-invariant metric on $G/K$. Denote by $\Delta $ the Laplace-Beltrami
operator on $X$ and by $d\left( .,.\right) $ the Riemannian distance and by $%
dx$ the associated measure on $X$.

Let $\Gamma $ be a discrete torsion free subgroup of $G$. Then the locally
symmetric space $M=\Gamma \backslash X$, equipped with the projection of the
canonical Riemannian structure of $X$, becomes a Riemannian manifold. We
denote also by $\Delta $ the laplacian on $M$, by $d(.,.)$ the Riemannian
distance and by $dx$ the associated measure on $M$. It is important to note
that in general, locally symmetric spaces have not bounded geometry since
the injectivity radius of $M$ is not in general strictly positive, \cite%
{DAMA}.

Fix $\mathfrak{a}$ a maximal abelian subspace of $\mathfrak{p}$ and denote
by $\mathfrak{a}^{\ast }$ the real dual of $\mathfrak{a}$. If $\dim 
\mathfrak{a}=1$, we say that $X$ has \textit{rank one}. Let $\Sigma \subset 
\mathfrak{a}^{\ast }$, be the root system of ($\mathfrak{g}$, $\mathfrak{a}$%
). Denote by $W$ the Weyl group associated to $\Sigma $ and choose a set $%
\Sigma ^{+}$ of positives roots. Denote by $\rho $ the half sum of positive
roots counted with their multiplicities. Let $\mathfrak{a}^{+}\subset 
\mathfrak{a}$ be the corresponding positive Weyl chamber and let $\overline{%
\mathfrak{a}_{+}}$ be its closure. Set 
\begin{equation*}
\rho _{m}=\min_{H\in \overline{\mathfrak{a}_{+}},\text{ }\left\vert
H\right\vert =1}\rho \left( H\right) .
\end{equation*}%
Note that in the rank one case $\rho _{m}=\left\vert \rho \right\vert $.

Denote by $s_{t}$ the fundamental solution of the Schr\"{o}dinger equation
on the symmetric space $X$: 
\begin{equation*}
i\partial _{t}s_{t}\left( x,y\right) =\Delta s_{t}\left( x,y\right) ,\text{
\ }t\in \mathbb{R}\text{, \ }x,y\in X\text{.}
\end{equation*}%
Then $s_{t}$ is a $K$-bi-invariant function and the Schr\"{o}dinger operator 
$S_{t}=e^{it\Delta }$ on $X$ is defined as a convolution operator:

\begin{equation}
S_{t}f(x)=\int_{G}f(y)s_{t}(y^{-1}x)dy=\left( f\ast s_{t}\right) \left(
x\right) ,\text{ \ \ }f\in C_{0}^{\infty }(X).  \label{so}
\end{equation}

Using that $s_{t}$ is $K$-bi-invariant, we deduce that if $f\in
C_{0}^{\infty }(M)$, then $S_{t}f$ is right $K$-invariant and left $\Gamma $%
-invariant i.e. a function on the locally symmetric space $M$. Thus the Schr%
\"{o}dinger operator $\widehat{S}_{t}$ on $M$ is also defined by formula (%
\ref{so}).

The first ingredient for the proof of the dispersive estimate (\ref{diest000}%
) are precise estimates of the kernel $s_{t}$. In the present work we deal
with the cases when either $X$ has rank one or $G$ is complex. The reason is
that in these cases the expression of the spherical Fourier transform allow
us to obtain precise estimates of kernel $s_{t}$. They are obtained in \cite[%
Section 3]{ANPIVA} in the context of rank one symmetric spaces and in
Section \ref{kern} in the case $G$ complex.

Set 
\begin{equation}
\widehat{s}_{t}(x,y)=\sum_{\gamma \in \Gamma }s_{t}(x,\gamma y).
\label{ker0}
\end{equation}%
Recall that the critical exponent $\delta \left( \Gamma \right) $ is defined
by 
\begin{equation*}
\delta \left( \Gamma \right) =\inf \left\{ \alpha >0:P_{\alpha }\left(
x,y\right) <\infty \right\} ,
\end{equation*}%
where 
\begin{equation*}
P_{\alpha }\left( x,y\right) =\sum_{\gamma \in \Gamma }e^{-\alpha d\left(
x,\gamma y\right) }
\end{equation*}%
are the Poincar\'{e} series, \cite{WE}. Note that $\delta \left( \Gamma
\right) \leq 2\left\vert \rho \right\vert $.

In Section \ref{SCK}, we show that the series (\ref{ker0}) converges when $%
\delta \left( \Gamma \right) <\rho _{m}$ and that $\widehat{S}_{t}$ is an
integral operator on $M$ with kernel $\widehat{s}_{t}(x,y)$: 
\begin{equation}
\widehat{S}_{t}f(x)=\int_{M}f(y)\widehat{s}_{t}(x,y)dy.  \label{kap}
\end{equation}

The expression (\ref{kap}) as well as the norm estimates of $\widehat{s}_{t}$
obtained in the same section are the second ingredient for the proof of the
dispersive estimate (\ref{diest000}) of the operator $\widehat{S}_{t}$.

The third ingredient is the following analogue of Kunze and Stein phenomenon
on locally symmetric spaces, proved in \cite{LOMAjga}, and presented in
detail in Section \ref{KSPH}. Let $\lambda _{0}$ be the bottom of the $L^{2}$%
-spectrum of $-\Delta $ on $M$. Then in \cite{LOMAjga} it is proved that
there exists a vector $\eta _{\Gamma }$ on the Euclidean sphere $S\left(
0,\left( \left\vert \rho \right\vert ^{2}-\lambda _{0}\right) ^{1/2}\right) $
of $\mathfrak{a}^{\ast }$, such that for all $p\in \left( 1,\infty \right) $
and for every $K$-bi-invariant function $\kappa $, the convolution operator $%
\ast \left\vert \kappa \right\vert $ with kernel $\left\vert \kappa
\right\vert $ satisfies the estimate%
\begin{equation}
\left\Vert \ast \left\vert \kappa \right\vert \right\Vert
_{L^{p}(M)\rightarrow L^{p}(M)}\leq \int_{G}\left\vert \kappa \left(
g\right) \right\vert \varphi _{-i\eta _{\Gamma }}\left( g\right) ^{s\left(
p\right) }dg,  \label{ks00}
\end{equation}%
where $\varphi _{\lambda }$ is the spherical function with index $\lambda $
and 
\begin{equation*}
s\left( p\right) =2\min \left( \left( 1/p\right) ,\left( 1/p^{\prime
}\right) \right) .
\end{equation*}

Note that Leuzinger \cite{LE} proved that if $G$ has no compact factors, its
center is trivial and $\delta \left( \Gamma \right) \leq \rho _{m}$, then $%
\lambda _{0}=\left\vert \rho \right\vert ^{2}$. In this case, $\eta _{\Gamma
}=0$ in (\ref{ks00}). Note that this is also the case if $M$ has rank one
and $\delta \left( \Gamma \right) \leq \rho _{m}$. For simplicity we shall
assume that \textit{all locally symmetric spaces} we deal with, satisfy $%
\lambda _{0}=\left\vert \rho \right\vert ^{2}$.

\begin{definition}
We say that the locally symmetric space $M=\Gamma \backslash G/K$ \textit{%
belongs in the class }$(S)$ if

(i) $\delta \left( \Gamma \right) <\rho _{m}$, and

(ii) for every $K$-bi-invariant function $\kappa $ the estimate (\ref{ks00})
is satisfied with $\eta _{\Gamma }=0$.
\end{definition}

In particular, if $M\in (S)$ and either $M$ has rank one or $G$ is complex,
we say that $M$ belongs in the class $(S_{0})$. Note that if $M\in (S_{0})$,
then the three ingredients for the proof of the dispersive estimate (\ref%
{diest000}) of the operator $\widehat{S}_{t}$ are available. In the present
work we treat the case $M\in (S_{0})$.

As it is explained in Section \ref{KSPH}, the estimate (\ref{ks00}) is valid
if $M$ belongs in one of the following three classes:

(i). $\Gamma $ \textit{is a lattice }i.e. $\limfunc{vol}\left( \Gamma
\backslash G\right) <\infty $,

(ii). $G$ \textit{possesses Kazhdan's property (T)}. Recall that $G$ has
property (T) iff $G$ has no simple factors locally isomorphic to $SO(n,1)$
or $SU(n,1)$, \cite[ch. 2]{HAVA}. In this case $\Gamma \backslash G/K\in
\left( S\right) $ for all discrete subgroups $\Gamma $ of $G$ with $\delta
\left( \Gamma \right) <\left\vert \rho \right\vert $.

Recall also that non-compact rank one symmetric spaces are the real, complex
and quaternionic hyperbolic spaces, denoted $H^{n}\left( \mathbb{R}\right) $%
, $H^{n}\left( \mathbb{C}\right) $, $H^{n}\left( \mathbb{H}\right) $
respectively, and the octonionic hyperbolic plane $H^{2}\left( \mathbb{O}%
\right) $. They have the following representation as quotients: 
\begin{eqnarray*}
H^{n}\left( \mathbb{R}\right) &=&SO(n,1)/SO(n),\text{ \ }H^{n}\left( \mathbb{%
C}\right) =SU(n,1)/SU(n), \\
H^{n}\left( \mathbb{H}\right) &=&Sp(n,1)/Sp(n)\text{, and }H^{2}\left( 
\mathbb{O}\right) =F_{4}^{-20}/Spin\left( 9\right) \mathbf{.}
\end{eqnarray*}%
So, $\Gamma \backslash H^{n}\left( \mathbb{H}\right) $ and $\Gamma
\backslash H^{2}\left( \mathbb{O}\right) $ belong in the class $(S)$ for 
\textit{all} discrete subgroups $\Gamma $ of $Sp\left( n,1\right) $ and $%
F_{4}^{-20}$ respectively such that $\delta \left( \Gamma \right)
<\left\vert \rho \right\vert $.

(iii) all quotients $\Gamma \backslash H^{n}\left( \mathbb{R}\right) $ and $%
\Gamma \backslash H^{n}\left( \mathbb{C}\right) $ with $\Gamma $ \textit{%
amenable} with $\delta \left( \Gamma \right) <\left\vert \rho \right\vert $,
even if they have infinite volume.

\subsection{Dispersive and Strichartz estimates on locally symmetric spaces}

The main result of the present paper is the following dispersive estimate.

\begin{theorem}
\label{diest13}Assume that $M\in (S_{0})$. If $q,\tilde{q}\in \left(
2,\infty \right] $, then for $|t|<1$, 
\begin{equation*}
\Vert \widehat{S}_{t}\Vert _{L^{\tilde{q}^{\prime }}(M)\rightarrow
L^{q}(M)}\leq c|t|^{-n\max {\{}\left( 1/2\right) -\left( 1/q\right) {,\left(
1/2\right) -\left( 1/\widetilde{q}\right) \}}},
\end{equation*}%
while for $|t|\geq 1$, 
\begin{equation*}
\Vert \widehat{S}_{t}\Vert _{L^{\tilde{q}^{\prime }}(M)\rightarrow
L^{q}(M)}\leq c|t|^{-3/2}\text{,}
\end{equation*}%
if $M$ has rank one, and 
\begin{equation*}
\Vert \widehat{S}_{t}\Vert _{L^{\tilde{q}^{\prime }}(M)\rightarrow
L^{q}(M)}\leq c|t|^{-n/2},\text{ }
\end{equation*}%
if $G$ is complex.
\end{theorem}

Consider the following Cauchy problem for the linear inhomogeneous Schr\"{o}%
dinger equation on $M$: 
\begin{equation}
\left\{ 
\begin{array}{l}
i\partial _{t}u(t,x)+\Delta u(t,x)=F(t,x), \\ 
u(0,x)=f(x).%
\end{array}%
\right.  \label{ilcp}
\end{equation}

Combining the above dispersive estimate with the classical $TT^{\ast }$
method developed by Kato \cite{KA}, Ginibre and Velo \cite{GIVE} and Keel
and Tao \cite{KETA}, we obtain Strichartz estimates for the solutions $%
u(t,x) $\ of (\ref{ilcp}). Consider the triangle

\begin{equation}
T_{n}=\left\{ \left( \tfrac{1}{p},\tfrac{1}{q}\right) \in \left( 0,\tfrac{1}{%
2}\right] \times \left( 0,\tfrac{1}{2}\right) :\tfrac{2}{p}+\tfrac{n}{q}\geq 
\tfrac{n}{2}\right\} \cup \left\{ \left( 0,\tfrac{1}{2}\right) \right\} .
\label{tr0}
\end{equation}

\begin{theorem}
\label{diest140}Assume that $M\in (S_{0})$. If $\left( p,q\right) $ and $%
\left( \tilde{p},\tilde{q}\right) $ are admissible pairs in the triangle $%
T_{n}$, then there exists a constant $c>0$ such that the following
Strichartz estimate holds for the solutions $u\left( t,x\right) $ of the
Cauchy problem (\ref{ilcp}): 
\begin{equation}
\Vert u\Vert _{L_{t}^{p}L_{x}^{q}}\leq c\left\{ \Vert f\Vert
_{L_{x}^{2}}+\Vert F\Vert _{L_{t}^{\tilde{p}^{\prime }}L_{x}^{\tilde{q}%
^{\prime }}}\right\} .  \label{stri0}
\end{equation}
\end{theorem}

As it is noticed in \cite{ANPI,ANPIVA} the above set $T_{n}$ of admissible
pairs is much wider that the admissible set in the case of $\mathbb{R}^{n}$
which is just the lower edge of the triangle. This phenomenon was already
observed for hyperbolic spaces in \cite{BACASTA, IONSTA}.

The paper is organised as follows. In Section 2 we describe the geometric
context of symmetric spaces and we present the spherical Fourier transform.
In Section \ref{KSPH} we present an analogue of Kunze and Stein phenomenon
for convolution operators on a class of locally symmetric spaces proved in 
\cite{LOMAjga} by Lohou\'{e} and Marias. In Section \ref{kern} we prove
pointwise estimates of the Schr\"{o}dinger kernel on symmetric spaces $X=G/K$
when $G$\ is complex. In Section \ref{SCK} we study the Schr\"{o}dinger
operator on $M$ and we prove norm estimates for its kernel.\ In Section \ref%
{diest} we prove dispersive estimates for the Schr\"{o}dinger operator on $M$
and we give the proofs of Theorems \ref{diest13} and \ref{diest140}.
Finally, we apply Strichartz estimates to study well-posedness and
scattering for NLS equations.

\section{Preliminaries}

In this section we shall recall some basic facts about symmetric spaces of
noncompact type. For more details see \cite{HEL,AN,ANLO,ION1,ION2,LOMAjga}.

Let $A$ be the analytic subgroup of $G$ with Lie algebra $\mathfrak{a}$. Let 
$\mathfrak{a}^{+}\subset \mathfrak{a}$ be a positive Weyl chamber and let $%
\overline{\mathfrak{a}_{+}}$ be its closure. Put $A^{+}=\exp \mathfrak{a}_{+}
$. Its closure in $G$ is $\overline{A^{+}}=\exp \overline{\mathfrak{a}_{+}}$%
. We have the Cartan decomposition 
\begin{equation*}
G=K(\overline{A^{+}})K=K(\exp \overline{\mathfrak{a}_{+}})K.
\end{equation*}%
Let $k_{1}$, $k_{2}$ and $\exp H$ be the components of $g\in G$ in $K$ and $%
\exp \mathfrak{a}$ according to the Cartan decomposition. Then $g$ is
written as $g=k_{1}\left( \exp H\right) k_{2}$. According to Cartan
decomposition, the Haar measure on $G$ is written as 
\begin{equation}
\int_{G}f\left( g\right) dg=c\int_{K}dk_{1}\int_{\mathfrak{a}^{+}}\delta
\left( H\right) dH\int_{K}f\left( k_{1}\left( \exp H\right) k_{2}\right)
dk_{2},  \label{mes}
\end{equation}%
where 
\begin{equation*}
\delta \left( H\right) =\prod_{\alpha \in \Sigma ^{+}}\sinh ^{m_{\alpha
}}\alpha \left( H\right) ,
\end{equation*}%
with $m_{\alpha }=\dim \mathfrak{g}_{\alpha }$, and 
\begin{equation*}
\mathfrak{g}_{\alpha }=\left\{ X\in \mathfrak{g}:\left[ H,X\right] =\alpha
\left( H\right) X\text{ for all }H\in \mathfrak{a}\right\} ,
\end{equation*}%
is the root space associated to the root $\alpha \in \Sigma ^{+}$. Note that 
\begin{equation}
\delta \left( H\right) \leq ce^{2\rho \left( H\right) }\text{, }\ \ H\in 
\mathfrak{a}_{+}.  \label{mod}
\end{equation}

If $G$ has real rank one, then it is well known \cite{ION1} that the root
system $\Sigma $ is either of the form $\left\{ -\alpha ,\alpha \right\} $
or of the form $\left\{ -\alpha ,-2\alpha ,\alpha ,2\alpha \right\} $. Thus $%
\rho =\alpha /2$ or $\rho =\left( 3/2\right) \alpha $. Thus $\rho
_{m}=\min_{H\in \overline{\mathfrak{a}_{+}},\text{ }\left\vert H\right\vert
=1}\rho \left( H\right) =\rho $.

Let $H_{0}$ be the unique element of $\mathfrak{a}$ with the property that $%
\alpha (H_{0})=1$. Set $a(s)=\exp (sH_{0})$, $s\in \mathbb{R}.$ Then $a:%
\mathbb{R}\longrightarrow A$ is a diffeomorphism and we identify $A=\exp 
\mathfrak{a}$ with $\mathbb{R}$ via $a$. We also normalize the Killing form
on $\mathfrak{g}$ such that 
\begin{equation*}
d(a(s).\mathbf{0},\mathbf{0})=|s|,\text{ \ for all }s\in \mathbb{R},
\end{equation*}%
where $\mathbf{0}=\left\{ K\right\} $ is the origin of $X$.

In this case we have that $A_{+}=\left\{ a(s):s\geq 0\right\} \simeq \mathbb{%
R}^{+}$ and any $K$-bi-invariant function $\kappa $ is identified with the
function $\kappa :\mathbb{R}_{+}\longrightarrow \mathbb{C}$, given by $%
\kappa (s)=\kappa (a(s).\mathbf{0})$. So, if $\kappa $ is $K$-bi-invariant,
then from (\ref{mes}) one has 
\begin{equation}
\int_{G}\kappa \left( g\right) dg\leq c\int_{\mathbb{R}^{+}}\kappa \left(
s\right) e^{2\left\vert \rho \right\vert s}ds.  \label{cart2}
\end{equation}

\subsection{The spherical Fourier transform}

Denote by $S\left( K\backslash G/K\right) $ the Schwartz space of $K$%
-bi-invariant functions on $G$. The spherical Fourier transform $\mathcal{H}$%
\ is defined by 
\begin{equation*}
\mathcal{H}f\left( \lambda \right) =\int_{G}f\left( x\right) \varphi
_{\lambda }\left( x\right) dx,\text{ \ }\lambda \in \mathfrak{a}^{\ast },\ \
f\in S\left( K\backslash G/K\right) ,
\end{equation*}%
where $\varphi _{\lambda }\left( x\right) $ are the elementary spherical
functions on $G$, which by a theorem of Harish-Chandra, \cite[p.418]{HEL},
are given by 
\begin{equation}
\varphi _{\lambda }\left( x\right) =\int_{K}e^{\left( \pm i\lambda +\rho
\right) \left( H\left( x^{\pm }k\right) \right) }dk\text{, \ }x\in G\text{,
\ }\lambda \in \mathfrak{a}_{\mathbb{C}}^{\ast }.  \label{phi0}
\end{equation}

Let $S\left( \mathfrak{a}^{\ast }\right) $ be the usual Schwartz space on $%
\mathfrak{a}^{\ast }$, and let us denote by $S\left( \mathfrak{a}^{\ast
}\right) ^{W}$ the subspace of $W$-invariants in $S\left( \mathfrak{a}^{\ast
}\right) $. Then, by a celebrated theorem of Harish-Chandra, $\mathcal{H}$
is an isomorphism between $S\left( K\backslash G/K\right) $ and $S\left( 
\mathfrak{a}^{\ast }\right) ^{W}$ and its inverse is given by 
\begin{equation}
\left( \mathcal{H}^{-1}f\right) \left( x\right) =c\int_{\mathfrak{a}^{\ast
}}f\left( \lambda \right) \varphi _{-\lambda }\left( x\right) \frac{d\lambda 
}{\left\vert \mathbf{c}\left( \lambda \right) \right\vert ^{2}},\ \ x\in G,\
f\in S\left( \mathfrak{a}^{\ast }\right) ^{W},  \label{invtran}
\end{equation}%
where $\mathbf{c}\left( \lambda \right) $ is the Harish-Chandra function.

Recall that the Schr\"{o}dinger kernel $s_{t}$ on $X$ is given by 
\begin{equation*}
s_{t}(\exp H)=\left( \mathcal{H}^{-1}w_{t}\right) \left( \exp H\right) \text{%
, \ }H\in \overline{\mathfrak{a}_{+}}
\end{equation*}%
where 
\begin{equation*}
w_{t}\left( \lambda \right) =e^{it\left\vert \rho \right\vert
^{2}}e^{it\left\vert \lambda \right\vert ^{2}},\text{ \ }\lambda \in 
\mathfrak{p}^{\ast }.
\end{equation*}

So, to obtain pointwise estimates of the kernel $s_{t}$ which are crucial
for our proofs, we need a manipulable expression of the inverse spherical
Fourier transform $\mathcal{H}^{-1}$. This happens exactly in the two cases
we deal with in the present work: the rank one case and the case $G$
complex. The case of rank one symmetric spaces is treated in \cite{ANPIVA}.
In Section \ref{kern} we treat the case $G$ complex by exploiting the fact
that the spherical Fourier transform boils down to the Euclidean Fourier
transform on $\mathfrak{p}$. More precisely, the inverse spherical Fourier
transform is given by the following formula: 
\begin{equation}
\mathcal{H}^{-1}f\left( \exp H\right) =c\varphi _{0}\left( \exp H\right)
\int_{\mathfrak{p}^{\ast }}f\left( \lambda \right) e^{i\lambda \left(
H\right) }d\lambda ,\text{ \ \ }H\in \mathfrak{p},  \label{cft}
\end{equation}%
where $\varphi _{0}$ is the basic spherical function \cite[p. 1312]{ANLO}.

\section{\label{KSPH}An analogue of Kunze-Stein's phenomenon on locally
symmetric spaces}

Let us recall that a central result in the theory of convolution operators
on semisimple Lie groups is the Kunze-Stein phenomenon, which states that if 
$p\in \lbrack 1,2)$, $f\in L^{2}(G)$ and $\kappa \in L^{p}(G)$, then 
\begin{equation}
||f\ast \kappa ||_{L^{2}(G)}\leq C\left( p\right) ||f||_{L^{2}(G)}||\kappa
||_{L^{p}(G)},  \label{ks}
\end{equation}%
(see \cite[p. 3361]{ION2}). This inequality was proved by Kunze and Stein 
\cite{KS} in the case when $G=SL(2,\mathbb{R})$ and by Cowling \cite{CO} in
the general case. In \cite{HE} Herz noticed that the inequality (\ref{ks})
can be sharpened if the kernel $\kappa $ is $K$-bi-invariant. Indeed, the
Herz's criterion \cite{HE} asserts that if $p\geq 1$ and $\kappa $ is a $K$%
-bi-invariant kernel, then 
\begin{eqnarray}
||\ast \left\vert \kappa \right\vert ||_{L^{p}(G)\rightarrow L^{p}(G)} &\leq
&C\int_{G}\left\vert \kappa (g)\right\vert \varphi _{-i\rho _{p}}(g)dg 
\notag \\
&=&C\int_{\mathfrak{a}_{+}}\left\vert \kappa (\exp H)\right\vert \varphi
_{-i\rho _{p}}(\exp H)\delta (H)dH,  \label{ks3}
\end{eqnarray}%
where $\rho _{p}=|2/p-1|\rho $, $p\geq 1$.

Note that for $p=2$, the best we can obtain in the Euclidean case is the
inequality 
\begin{equation*}
||\ast \left\vert \kappa \right\vert ||_{L^{2}(\mathbb{R}^{n})\rightarrow
L^{2}(\mathbb{R}^{n})}\leq ||\kappa ||_{L^{1}(\mathbb{R}^{n})},
\end{equation*}%
while in the semisimple case we have that 
\begin{equation}
||\ast \left\vert \kappa \right\vert ||_{L^{2}(G)\rightarrow L^{2}(G)}\leq
C\int_{\mathfrak{a}_{+}}\left\vert \kappa (\exp H)\right\vert \varphi
_{0}(\exp H)\delta (H)dH.  \label{ks000}
\end{equation}%
Bearing in mind that 
\begin{equation*}
||\kappa ||_{L^{1}(G)}=\int_{\mathfrak{a}_{+}}\left\vert \kappa (\exp
H)\right\vert \delta (H)dH,
\end{equation*}%
we deduce from the above norm estimates that for $p=2$, the non-trivial gain
over the Euclidean case is the factor $\varphi _{0}(\exp H)$.

Kunze-Stein's phenomenon is no more valid on locally symmetric spaces $%
M=\Gamma \backslash G/K$. In \cite{LOMAjga} Lohou\'{e} et Marias proved an
analogue of this phenomenon for a class of locally symmetric spaces. More
precisely, let $\lambda _{0}$ be the bottom of the $L^{2}$-spectrum of $%
-\Delta $ on $M$.\ We say that $M$ \textit{possesses property (KS)} if there
exists a vector $\eta _{\Gamma }$ on the Euclidean sphere $S\left( 0,\left(
\left\vert \rho \right\vert ^{2}-\lambda _{0}\right) ^{1/2}\right) $ of $%
\mathfrak{a}^{\ast }$, such that for all $p\in \left( 1,\infty \right) $,

\begin{equation}
\left\Vert \ast \left\vert \kappa \right\vert \right\Vert
_{L^{p}(M)\rightarrow L^{p}(M)}\leq \int_{G}\left\vert \kappa \left(
g\right) \right\vert \varphi _{-i\eta _{\Gamma }}\left( g\right) ^{s\left(
p\right) }dg,\text{ }  \label{loma1}
\end{equation}%
where 
\begin{equation}
s\left( p\right) =2\min \left( \left( 1/p\right) ,\left( 1/p^{\prime
}\right) \right) .  \label{loma2}
\end{equation}

In \cite{LOMAjga} it is shown that $M$ possesses property (KS)\textit{\ }if
it is contained in the following three classes:

(i). $\Gamma $ \textit{is a lattice }i.e. $\limfunc{vol}\left( \Gamma
\backslash G\right) <\infty $,

(ii). $G$ \textit{possesses Kazhdan's property (T)}. Recall that $G$ has
property (T)\ iff $G$ has no simple factors locally isomorphic to $SO(n,1)$
or $SU(n,1)$, \cite[ch. 2]{HAVA}. In this case $\Gamma \backslash G/K$
possesses property (KS) for all discrete subgroups $\Gamma $ of $G$. Recall
that $H^{n}\left( \mathbb{H}\right) =Sp\left( n,1\right) /Sp\left( n\right) $
and $H^{2}\left( \mathbb{O}\right) =F_{4}^{-20}/Spin\left( 9\right) $. So, $%
\Gamma \backslash H^{n}\left( \mathbb{H}\right) $ and $\Gamma \backslash
H^{2}\left( \mathbb{O}\right) $ have property (KS) for \textit{all} discrete
subgroups $\Gamma $ of $Sp\left( n,1\right) $ and $F_{4}^{-20}$ respectively.

Thus, from cases (i) and (ii) we deduce that \textit{all} locally symmetric
spaces $\Gamma \backslash H^{n}\left( \mathbb{H}\right) $ and $\Gamma
\backslash H^{2}\left( \mathbb{O}\right) $ have property (KS).

On the contrary, the isometry groups $SO(n,1)$ and $SU(n,1)$ of real and
complex hyperbolic spaces do not have property (T) and consequently the
quotients $\Gamma \backslash H^{n}\left( \mathbb{R}\right) $ and $\Gamma
\backslash H^{n}\left( \mathbb{C}\right) $ of infinite volume do not in
general belong in the class (ii). The class (iii) below covers also this
case.

(iii) $\Gamma \backslash G$ \textit{is non-amenable}. Note that since $G$ is
non-amenable, then $\Gamma \backslash G$ is non-amenable if $\Gamma $ is
amenable. So, if $\Gamma $ is amenable, then the quotients $\Gamma
\backslash H^{n}\left( \mathbb{R}\right) $ and $\Gamma \backslash
H^{n}\left( \mathbb{C}\right) $ possesses property (KS) even if they have
infinite volume. Note that if $\Gamma $ is finitely generated and has
subexponential growth, then $\Gamma $ is amenable, \cite{ADS,GR}.

Let us now explain when a finitely generated group $\Gamma $ has
subexponential growth. Let $A=\{a_{1},a_{2},...,a_{m}\}$ be a system of
generators of $\Gamma $. The length $|g|_{A}$ of $g\in \Gamma $ with respect
to $A$ is the length $n$ of the shortest representation of $g$ in the form $%
g=a_{i_{1}}^{\pm }a_{i_{2}}^{\pm }\cdots a_{i_{n}}^{\pm }$, $a_{i_{j}}\in A$%
. This depends on the set $A$ but, for any two systems of generators $A$ and 
$B$, $|g|_{A}$ and $|g|_{B}$ are equivalent. The growth function of $\Gamma $
with respect to the set $A$ is the function 
\begin{equation*}
\gamma _{\Gamma }^{A}\left( n\right) =\#\left\{ g\in G:|g|_{A}\leq n\right\}
,
\end{equation*}%
where $\#E$ denotes the cardinality of the set $E$. We say that $\Gamma $
has subexponential growth if $\gamma _{\Gamma }^{A}\left( n\right) $ grows
more slowly than any exponential function.

\section{\label{kern}Pointwise estimates of the Schr\"{o}dinger kernel on $X$%
}

As it is already mentioned, the Schr\"{o}dinger kernel $s_{t}$ on $X$ is
given by 
\begin{equation}
s_{t}(\exp H)=\left( \mathcal{H}^{-1}w_{t}\right) \left( \exp H\right) \text{%
, \ }H\in \overline{\mathfrak{a}_{+}},  \label{sft}
\end{equation}%
where 
\begin{equation*}
w_{t}\left( \lambda \right) =e^{it\left\vert \rho \right\vert
^{2}}e^{it\left\vert \lambda \right\vert ^{2}},\text{ \ }\lambda \in 
\mathfrak{p}^{\ast }.
\end{equation*}

Using (\ref{sft}) and the expression of the inverse spherical Fourier
transform $\mathcal{H}^{-1}$\ in the case of rank one symmetric spaces,
Anker, Pierfelice et Vallarino \cite{ANPIVA} obtained the following
estimates of $s_{t}$:

\begin{equation}
\left\vert s_{t}(r)\right\vert \leq c\psi _{1}\left( t,r\right)
e^{-\left\vert \rho \right\vert r},  \label{schroest11}
\end{equation}%
where 
\begin{equation}
\psi _{1}\left( t,r\right) =\left\{ 
\begin{array}{l}
|t|^{-n/2}(1+r)^{\left( n-1\right) /2},\text{ \ if \ }\left\vert
t\right\vert \leq 1+r, \\ 
|t|^{-3/2}(1+r),\text{ \ if \ }\left\vert t\right\vert >1+r.%
\end{array}%
\right.  \label{psi0}
\end{equation}%
(See also \cite{ANPI,MANcam} for the case of the real hyperbolic space).

\begin{lemma}
If $G$ is complex, then the Schr\"{o}dinger kernel $s_{t}$ on $X=G/K$ is
given by 
\begin{equation}
s_{t}\left( \exp H\right) =c\varphi _{0}\left( \exp H\right)
t^{-n/2}e^{it\left\vert \rho \right\vert ^{2}}e^{-i\left\vert H\right\vert
^{2}/4t}\text{, \ }H\in \overline{\mathfrak{a}_{+}},\text{ \ }t\in \mathbb{R}%
\text{, }  \label{ft01}
\end{equation}%
where $n=$ $\dim \mathfrak{p}=\dim X$.
\end{lemma}

\begin{proof}
As it is already mentioned, if $G$ is complex, then the inverse spherical
Fourier transform is given by 
\begin{equation}
\mathcal{H}^{-1}f\left( \exp H\right) =c\varphi _{0}\left( \exp H\right)
\int_{\mathfrak{p}^{\ast }}f\left( \lambda \right) e^{i\lambda \left(
H\right) }d\lambda ,\text{ \ \ }H\in \mathfrak{p},  \label{ft00}
\end{equation}%
where $\varphi _{0}$ is the basic spherical function \cite[p. 1312]{ANLO}.
From (\ref{ft00}) and (\ref{sft}) we get that 
\begin{equation*}
s_{t}\left( \exp H\right) =c\varphi _{0}\left( \exp H\right) \int_{\mathfrak{%
p}^{\ast }}e^{it\left( \left\vert \rho \right\vert ^{2}+\left\vert \lambda
\right\vert ^{2}\right) }e^{i\lambda \left( H\right) }d\lambda ,\text{ \ }%
H\in \mathfrak{p.}
\end{equation*}%
Write 
\begin{equation*}
e^{it\left( \left\vert \rho \right\vert ^{2}+\left\vert \lambda \right\vert
^{2}\right) }e^{i\lambda \left( H\right) }=e^{it\left\vert \rho \right\vert
^{2}}\dprod\limits_{j\leq n}e^{i\lambda _{j}^{2}}e^{i\lambda _{j}H_{j}},
\end{equation*}%
and note that the function $\lambda \longrightarrow g\left( \lambda \right)
=e^{i\lambda ^{2}}e^{i\lambda H}$, $\lambda \in \mathbb{C}$, is analytic.
So, we can compute the integral 
\begin{equation*}
I\left( t,H\right) :=\int_{\mathbb{R}}e^{i\lambda ^{2}}e^{i\lambda
H}d\lambda 
\end{equation*}%
by changing the path of integration. In fact, we shall integrate on the
first diagonal $\gamma \left( \lambda \right) =e^{i\pi /4}\lambda $, $%
\lambda \in \mathbb{R}$, and get that 
\begin{eqnarray*}
I\left( t,H\right)  &=&\int_{\mathbb{R}}e^{ite^{i\left( \pi /2\right)
}\lambda ^{2}}e^{ie^{i\left( \pi /4\right) }\lambda H}d\lambda =\int_{%
\mathbb{R}}e^{-t\lambda ^{2}}e^{i\lambda \left( e^{i\left( \pi /4\right)
}H\right) }d\lambda  \\
&=&\left( 4\pi t\right) ^{-1/2}e^{-\left( e^{i\left( \pi /4\right) }H\right)
^{2}/4t}=\left( 4\pi t\right) ^{-1/2}e^{-iH^{2}/4t},
\end{eqnarray*}%
and (\ref{ft01}) follows by bearing in mind that $s_{t}$ is $K$-bi-invariant.
\end{proof}

\begin{proposition}
\label{comest}If $G$ is complex, then the Schr\"{o}dinger kernel $s_{t}$ on $%
X=G/K$ satisfies the following estimates 
\begin{equation}
\left\vert s_{t}\left( \exp H\right) \right\vert \leq ct^{-n/2}\left(
1+\left\vert H\right\vert \right) ^{\alpha }e^{-\rho _{m}\left\vert
H\right\vert }\text{, \ }H\in \overline{\mathfrak{a}_{+}},\text{ \ }t\in 
\mathbb{R}\text{,}  \label{ft03}
\end{equation}%
for some constants $c$, $a>0$.
\end{proposition}

\begin{proof}
Recall that for $H\in \overline{\mathfrak{a}_{+}}$, 
\begin{equation}
\varphi _{0}\left( \exp H\right) \leq c\left( 1+\left\vert H\right\vert
\right) ^{a}e^{-\rho \left( H\right) },  \label{phi00}
\end{equation}%
for some constants $c$, $a>0$, \cite{HEL}. Combining (\ref{phi00}) with (\ref%
{ft01}), we get 
\begin{eqnarray*}
\left\vert s_{t}\left( \exp H\right) \right\vert  &=&c\varphi _{0}\left(
\exp H\right) \leq c\left( 1+\left\vert H\right\vert \right) ^{a}e^{-\rho
\left( H\right) } \\
&\leq &ct^{-n/2}\left( 1+\left\vert H\right\vert \right) ^{a}e^{-\rho
_{m}\left\vert H\right\vert }.
\end{eqnarray*}
\end{proof}

\begin{remark}
As it is already mentioned, if $X$ has rank one, then $\left\vert \rho
\right\vert =\rho _{m}$. Also, in both cases, $s_{t}$ is $K$-bi-invariant.
So. the estimates (\ref{schroest11}) and (\ref{ft03}) of $s_{t}$ can be
written in the following form: 
\begin{equation}
\left\vert s_{t}\left( \exp H\right) \right\vert \leq c\psi \left(
t,H\right) e^{-\rho _{m}\left\vert H\right\vert },\text{ \ }H\in \overline{%
\mathfrak{a}_{+}},\text{ \ }t\in \mathbb{R}\text{,}  \label{kest1}
\end{equation}%
where 
\begin{equation}
\psi \left( t,H\right) =\left\{ 
\begin{array}{l}
\psi _{1}\left( t,H\right) \text{, \ if }M\text{ has rank one,} \\ 
\psi _{2}\left( t,H\right) :=t^{-n/2}\left( 1+\left\vert H\right\vert
\right) ^{a}\text{, \ if }G\text{ is complex,}%
\end{array}%
\right.  \label{psi2}
\end{equation}%
and $\psi _{1}$ is defined in (\ref{psi0}).
\end{remark}

\section{\label{SCK}Norm estimates of the Schr\"{o}dinger kernel on $M$}

Recall that the Schr\"{o}dinger operator $\widehat{S}_{t}$ on $M$ is
initially defined as a convolution operator

\begin{equation}
\widehat{S}_{t}f(x)=\int_{G}s_{t}(y^{-1}x)f(y)dy,\text{ \ \ }f\in
C_{0}^{\infty }(M).  \label{schroop1}
\end{equation}

Set $s_{t}(x,y)=s_{t}(y^{-1}x)$ and 
\begin{equation}
\widehat{s}_{t}\left( x,y\right) =\sum_{\gamma \in \Gamma }s_{t}\left(
x,\gamma y\right) =\sum_{\gamma \in \Gamma }s_{t}\left( \left( \gamma
y\right) ^{-1}x\right) .  \label{kerco}
\end{equation}

\begin{proposition}
\label{kest0}For all groups $\Gamma $ with $\delta \left( \Gamma \right)
<\rho _{m}$, the series (\ref{kerco}) converges and the Schr\"{o}dinger
operator $\widehat{S}_{t}$ on $M$ is given by 
\begin{equation}
\widehat{S}_{t}f\left( x\right) =\int_{M}\widehat{s}_{t}\left( x,y\right)
f\left( y\right) dy.  \label{kerco1}
\end{equation}
\end{proposition}

\begin{proof}
Use the Cartan decomposition and write $\left( \gamma y\right)
^{-1}x=k_{\gamma }\exp H_{\gamma }k_{\gamma }^{\prime }$. Then, since $s_{t}$
is $K$-bi-invariant, $s_{t}\left( \left( \gamma y\right) ^{-1}x\right)
=s_{t}\left( \exp H_{\gamma }\right) $ and the estimate (\ref{kest1})
implies that 
\begin{eqnarray*}
\left\vert \widehat{s}_{t}\left( x,y\right) \right\vert  &\leq &\sum_{\gamma
\in \Gamma }\left\vert s_{t}\left( \left( \gamma y\right) ^{-1}x\right)
\right\vert \leq \sum_{\gamma \in \Gamma }\left\vert s_{t}\left( \exp
H_{\gamma }\right) \right\vert  \\
&\leq &c\sum_{\gamma \in \Gamma }\psi \left( t,\left\vert H_{\gamma
}\right\vert \right) e^{-\rho _{m}\left\vert H_{\gamma }\right\vert }.
\end{eqnarray*}%
If $M$ has rank one, then by (\ref{psi2}) and (\ref{psi0}) we have that for
every $\varepsilon >0$ 
\begin{multline*}
\left\vert \widehat{s}_{t}\left( x,y\right) \right\vert \leq \sum_{\gamma
\in \Gamma }\psi \left( t,\left\vert H_{\gamma }\right\vert \right) e^{-\rho
_{m}\left\vert H_{\gamma }\right\vert } \\
\leq c|t|^{-n/2}\sum_{\{\gamma \in \Gamma :\left\vert t\right\vert \leq
1+\left\vert H_{\gamma }\right\vert \}}e^{(\varepsilon -\rho _{m})\left\vert
H_{\gamma }\right\vert }+c|t|^{-3/2}\sum_{\{\gamma \in \Gamma :\left\vert
t\right\vert >1+\left\vert H_{\gamma }\right\vert \}}e^{(\varepsilon -\rho
_{m})\left\vert H_{\gamma }\right\vert } \\
\leq c|t|^{-n/2}\sum_{\gamma \in \Gamma }e^{(\varepsilon -\rho _{m})d\left(
0,\left( \gamma y\right) ^{-1}x\right) }+c|t|^{-3/2}\sum_{\gamma \in \Gamma
}e^{(\varepsilon -\rho _{m})d\left( 0,\left( \gamma y\right) ^{-1}x\right) }
\\
\leq c(|t|^{-3/2}+|t|^{-n/2})\sum_{\gamma \in \Gamma }e^{(\varepsilon -\rho
_{m})d\left( x,\gamma y\right) } \\
\leq c(|t|^{-3/2}+|t|^{-n/2})P_{\rho _{m}-\varepsilon }\left( x,y\right)
<\infty 
\end{multline*}%
provided that $\rho _{m}>\delta \left( \Gamma \right) $.

The proof of the case $G$ complex is similar and then omitted.

Let us now prove (\ref{kerco1}). Since $s_{t}$ and $f$ are right $K$%
-invariant, from (\ref{schroop1}) we get that 
\begin{equation*}
\widehat{S}_{t}f(x)=\int_{X}s_{t}(x,y)f(y)dy.
\end{equation*}%
Now, since $f$ is left $\Gamma $-invariant, by Weyl's formula we find that \ 
\begin{eqnarray*}
\widehat{S}_{t}(f)(x) &=&\int_{X}f(y)s_{t}(x,y)dy=\int_{\Gamma \backslash
X}\left( \sum_{\gamma \in \Gamma }f(\gamma y)s_{t}(x,\gamma y)\right) dy \\
&=&\int_{M}f(y)\widehat{s}_{t}(x,y)dy.
\end{eqnarray*}
\end{proof}

Next, we prove norm estimates for the Schr\"{o}dinger kernel on $M$ which
are crucial for the proofs of our results.

\begin{proposition}
\label{kest2}If $\delta (\Gamma )<\rho _{m}$, then for any $q>2$ and $x\in X$%
, 
\begin{eqnarray}
\Vert \widehat{s}_{t}(x,.)\Vert _{L^{q}(M)} &\leq &c\Vert {s}_{t}(x,.)\Vert
_{L^{q}(X)}  \notag \\
&\leq &c\Psi \left( t\right) ,\text{ \ }t\in \mathbb{R}\text{,}
\label{intest1}
\end{eqnarray}%
where 
\begin{equation*}
\Psi \left( t\right) =\left\{ 
\begin{array}{c}
\left\vert t\right\vert ^{-n/2},\text{ \ \ if \ \ }|t|\leq 1, \\ 
|t|^{-3/2},\text{ \ \ if \ \ }|t|>1,%
\end{array}%
\right.
\end{equation*}%
in the rank one case, and 
\begin{equation*}
\Psi \left( t\right) =\left\vert t\right\vert ^{-n/2},\text{ \ }t\in \mathbb{%
R}\text{,}
\end{equation*}%
in the case when $G$ is complex.
\end{proposition}

\begin{proof}
Fix $N_{0}\in \mathbb{N}$ and write $\Gamma =\Gamma _{1}\cup \Gamma _{2}$,
where 
\begin{equation*}
\Gamma _{1}=\{\gamma \in \Gamma :d(x,\gamma y)>N_{0}\}\text{ and }\Gamma
_{2}=\{\gamma \in \Gamma :d(x,\gamma y)\leq N_{0}\}.
\end{equation*}%
Since $\Gamma $ is a discrete group, then $\Gamma _{2}$ is a finite set. On
the other hand, since the series $\sum_{\gamma \in \Gamma }s_{t}(x,\gamma y)$
is convergent, we can choose $N_{0}$ such that 
\begin{equation*}
\left\vert \sum_{\gamma \in \Gamma }s_{t}(x,\gamma y)\right\vert \leq
2\left\vert \sum_{\gamma \in \Gamma _{2}}s_{t}(x,\gamma y)\right\vert .
\end{equation*}%
But $\Gamma _{2}$ is a finite set, so, for $q>2$, we have that 
\begin{equation*}
\left\vert \sum_{\gamma \in \Gamma _{2}}s_{t}(x,\gamma y)\right\vert
^{q}\leq c\sum_{\gamma \in \Gamma _{2}}\left\vert s_{t}(x,\gamma
y)\right\vert ^{q}.
\end{equation*}%
It follows that 
\begin{equation}
\left\vert \sum_{\gamma \in \Gamma }s_{t}(x,\gamma y)\right\vert ^{q}\leq
c\sum_{\gamma \in \Gamma _{2}}\left\vert s_{t}(x,\gamma y)\right\vert
^{q}\leq c\sum_{\gamma \in \Gamma }\left\vert s_{t}(x,\gamma y)\right\vert
^{q}.  \label{sum}
\end{equation}%
Using (\ref{sum}) and Weyl's formula, we have that 
\begin{equation*}
\begin{split}
\Vert \widehat{s}_{t}(x,.)\Vert _{L^{q}(M)}^{q}& \leq 2\int_{M}\left\vert
\sum_{\gamma \in \Gamma _{2}}s_{t}(x,\gamma y)\right\vert ^{q}dy\leq
c\int_{M}\sum_{\gamma \in \Gamma }|s_{t}(x,\gamma y)|^{q}dy \\
& =c\int_{X}|s_{t}(x,y)|^{q}dy=c\Vert {s}_{t}(x,.)\Vert _{L^{q}(X)}^{q}.
\end{split}%
\end{equation*}

Assume now that $M$ has rank one. Using that $s_{t}\left( x,y\right) $ is
radial, from (\ref{cart2}) we have that 
\begin{equation}
\left\Vert s_{t}\left( x,.\right) \right\Vert _{_{L^{q}(X)}}^{q}\leq c\int_{%
\mathbb{R}^{+}}\left\vert s_{t}\left( r\right) \right\vert
^{q}e^{2\left\vert \rho \right\vert r}dr.  \label{qest}
\end{equation}

Next, using the estimates (\ref{schroest11}) of ${s}_{t}(x,y)$ and the fact $%
q>2$, we shall prove (\ref{intest1}). Indeed, choose $\epsilon >0$ such that 
$\epsilon <(q-2)|\rho |$. Then if $\left\vert t\right\vert \geq 1$, by (\ref%
{qest}) we get 
\begin{multline*}
\Vert {s}_{t}(x,.)\Vert _{L^{q}(X)}^{q}=\int_{X}|s_{t}(x,y)|^{q}dy \\
\leq c\left\vert t\right\vert ^{-3q/2}\int_{0}^{\left\vert t\right\vert
-1}e^{r\left( \epsilon -q|\rho |+2|\rho |\right) }dr+c\left\vert
t\right\vert ^{-nq/2}\int_{\left\vert t\right\vert -1}^{\infty }e^{r\left(
\epsilon -q|\rho |+2|\rho |\right) }dr \\
\leq c\left( \left\vert t\right\vert ^{-3q/2}+\left\vert t\right\vert
^{-nq/2}\right) \int_{0}^{\infty }e^{-r\left( \left( q-2\right) |\rho
|-\epsilon \right) }dr \\
\leq c\left\vert t\right\vert ^{-3q/2}.
\end{multline*}

The case $\left\vert t\right\vert \leq 1$, as well as the case $G$ complex,
are similar and thus omitted.
\end{proof}

\section{\label{diest}Dispersive estimates for the Schr\"{o}dinger operator
on $M$.}

In this section we prove dispersive estimates for the Schr\"{o}dinger
operator $\widehat{S}_{t}$ on locally symmetric spaces in the class $(S_{0})$%
. To begin with, we make use of the estimates of the $L^{q}$-norm of the
kernel $\widehat{s}_{t}\left( x,y\right) $ obtained in the previous section,
in order to estimate the operator norms of $\widehat{S}_{t}$.

\begin{lemma}
\label{diest1}If $\delta (\Gamma )<\rho _{m}$, then for any $q>2$ 
\begin{equation}
\Vert \widehat{S}_{t}\Vert _{L^{1}\left( M\right) \rightarrow L^{q}\left(
M\right) }\leq c\Psi \left( t\right) ,\text{ \ \ }t\in \mathbb{R}\text{,}
\label{diest2}
\end{equation}%
and 
\begin{equation}
\Vert \widehat{S}_{t}\Vert _{L^{q^{\prime }}\left( M\right) \rightarrow
L^{\infty }\left( M\right) }\leq c\Psi \left( t\right) ,  \label{diest4}
\end{equation}%
where $\Psi \left( t\right) $ is defined in Proposition \ref{kest2}.
\end{lemma}

\begin{proof}
Assume that $G$ is complex. If $f\in L^{1}(M)$ and $q>2$, then Minkowski's
inequality and Proposition \ref{kest2} imply that 
\begin{eqnarray*}
\Vert \widehat{S}_{t}f\Vert _{L^{q}(M)} &=&\left( \int_{M}\left\vert \int_{M}%
\widehat{s}_{t}(x,y)f(y)dy\right\vert ^{q}dx\right) ^{1/q} \\
&\leq &\int_{M}\left( \int_{M}\left\vert \widehat{s}_{t}(x,y)f(y)\right\vert
^{q}dx\right) ^{1/q}dy \\
&=&\int_{M}|f(y)|\left( \int_{M}\widehat{s}_{t}(x,y)^{q}dx\right) ^{1/q}dy \\
&\leq &\Vert f\Vert _{L^{1}(M)}\sup_{x}\Vert s_{t}(x,.)\Vert _{L^{q}(X)} \\
&\leq &c|t|^{-n/2}\Vert f\Vert _{L^{1}(M)}.
\end{eqnarray*}%
The estimate (\ref{diest4}) follows from (\ref{diest2}) by duality. The rank
one case is similar.
\end{proof}

Next, we make use of the analogue of Kunze-Stein phenomenon for locally
symmetric spaces proved in \cite{LOMAjga} and presented in Section 2, in
order to obtain the estimate of the norm $\Vert \ast s_{t}\Vert
_{L^{q^{\prime }}(M)\rightarrow L^{q}(M)}$. For that, as in \cite[p. 988]%
{ANPIVA}, we consider the spaces $A_{q}=L^{q/2}\left( \overline{\mathfrak{a}%
_{+}},\varphi _{0}\delta \right) $, $q\in \left[ 2,\infty \right) $, of all $%
K$-bi-invariant functions $f$ on $G$, that satisfy 
\begin{equation*}
\Vert f\Vert _{A_{q}}^{q/2}=\int_{\overline{\mathfrak{a}_{+}}}|f(\exp {H}%
)|^{q/2}\varphi _{0}(\exp {H})\delta (H)dH<\infty .
\end{equation*}%
For $p=\infty $, we take $A_{\infty }=L^{\infty }\left( \overline{\mathfrak{a%
}_{+}}\right) $.

The following analogue of Theorem 4.2 of \cite{ANPIVA} is a consequence of
Kunze and Stein phenomenon.

\begin{theorem}
\label{diest5}Assume that $M\in (S)$. If $\kappa \in S\left( K\backslash
G/K\right) $, then for all $q\geq 2$, 
\begin{equation}
L^{q^{\prime }}(M)\ast {A_{q}}\subset {L^{q}(M)}.  \label{diest7}
\end{equation}%
In other words, there exists a $c_{q}>0$ such that 
\begin{equation}
\Vert f\ast \kappa \Vert _{L^{q}(M)}\leq c_{q}\Vert \kappa \Vert
_{A_{q}}\Vert f\Vert _{L^{q^{\prime }}(M)}.  \label{diest6}
\end{equation}
\end{theorem}

\begin{proof}
Since $\kappa \in A_{2}$, then from (\ref{loma1}) it follows that 
\begin{eqnarray}
\Vert \ast \kappa \Vert _{L^{2}(M)\rightarrow L^{2}(M)} &\leq &\int_{%
\overline{\mathfrak{a}_{+}}}\left\vert \kappa \left( \exp H\right)
\right\vert |\varphi _{0}(\exp {H})\delta (H)dH  \notag \\
&=&\Vert \kappa \Vert _{A_{2}}.  \label{diest8}
\end{eqnarray}%
Furthermore, $\kappa \in A_{\infty }$, so, for every $f\in L^{1}(M)$ and $%
x\in G$, we have 
\begin{equation*}
\left\vert f\ast \kappa (x)\right\vert \leq \int_{G}|\kappa (\overline{g}%
^{-1}x)||f(\overline{g})|dg\leq \Vert \kappa \Vert _{\infty }\Vert f\Vert
_{L^{1}(M)},\text{ }
\end{equation*}%
i.e. 
\begin{equation}
\Vert \ast \kappa \Vert _{L^{1}(M)\rightarrow L^{\infty }(M)}\leq \Vert
\kappa \Vert _{A_{\infty }}.  \label{diest9}
\end{equation}

From (\ref{diest8}) and (\ref{diest9}), it follows that 
\begin{equation*}
L^{2}(M)\ast A_{2}\subset {L^{2}(M),}
\end{equation*}%
and 
\begin{equation*}
L^{1}(M)\ast A_{\infty }\subset {L^{\infty }(M).}
\end{equation*}

By interpolating between the case $q=2$ and $q=\infty $, we obtain that for
any $\theta \in \left( 0,1\right) $ 
\begin{equation}
\lbrack L^{2}(M),L^{1}(M)]_{\theta }\ast {[A_{2},A_{\infty }]_{\theta }}%
\subset \lbrack L^{2}(M),L^{\infty }(M)]_{\theta }.  \label{diest100}
\end{equation}%
Choose $\theta =2/q<1$. Then 
\begin{equation}
\lbrack L^{2}(M),L^{1}(M)]_{\theta }=L^{p_{\theta }}(M)=L^{q^{\prime }}(M),
\label{diest11}
\end{equation}%
since $\tfrac{1}{p_{\theta }}=\tfrac{\theta }{2}+\tfrac{1-\theta }{1}=1-%
\tfrac{\theta }{2}=\tfrac{1}{q^{\prime }}.$

Similarly, 
\begin{equation}
\lbrack A_{2},A_{\infty }]_{\theta }=[L^{1}(\overline{\mathfrak{a}_{+}}%
,\varphi _{0}\delta ,),L^{\infty }(\overline{\mathfrak{a}_{+}})]_{\theta
}=L^{q/2}(\overline{\mathfrak{a}_{+}},\varphi _{0}\delta )=A_{q},
\label{diest12}
\end{equation}%
and%
\begin{equation}
\lbrack L^{2}(M),L^{\infty }(M)]_{\theta }=L^{q_{\theta }}(M)=L^{q}(M).
\label{diest110}
\end{equation}

Putting together (\ref{diest100}), (\ref{diest11}) and (\ref{diest12}) we
get that 
\begin{equation*}
L^{q^{\prime }}(M)\ast A_{q}\subset L^{q}(M).
\end{equation*}
\end{proof}

\begin{lemma}
\label{diest12a}Assume that $M\in (S_{0})$. If $q>2$, then 
\begin{equation*}
\left\Vert \ast s_{t}\right\Vert _{L^{q^{\prime }}(M)\rightarrow
L^{q}(M)}\leq c\left( q\right) \Psi \left( t\right) ,
\end{equation*}%
where $\Psi $ is defined in Proposition \ref{kest2}.
\end{lemma}

\begin{proof}
Assume that $M$ has rank one. Let $\omega _{R}$, $R\in \mathbb{R}$, be an
even $C^{\infty }$ cut-off function on $\mathbb{R}$, such that $\omega
_{R}\left( r\right) =1$, for $\left\vert r\right\vert \leq R$ and $\omega
_{R}\left( r\right) =0$, for $\left\vert r\right\vert >2R$.

Set $s_{t}^{R}\left( r\right) =\omega _{R}\left( r\right) s_{t}\left(
r\right) $. Then, from Theorem \ref{diest5}, it follows that for $q>2$ and
any $R>0$, we have that 
\begin{equation*}
\left\Vert \ast s_{t}^{R}\right\Vert _{L^{q^{\prime }}\left( M\right)
\rightarrow L^{q}\left( M\right) }\leq c\left\Vert s_{t}^{R}\right\Vert
_{A_{q}}=c\left( \int_{\mathbb{R}_{+}}\left\vert s_{t}^{R}(r)\right\vert
^{q/2}\varphi _{0}(r)\delta (r)dr\right) ^{2/q}.
\end{equation*}

Using the estimates (\ref{schroest11}) of $s_{t}\left( r\right) $, and the
estimates 
\begin{equation*}
\varphi _{0}(r)\leq c\left( 1+r\right) ^{\alpha }e^{-r\left\vert \rho
\right\vert }\text{ and }\delta (r)\leq e^{2r\left\vert \rho \right\vert },
\end{equation*}%
we get that for $q>2$ and $\left\vert t\right\vert >1$, 
\begin{equation*}
\left\Vert s_{t}^{R}\right\Vert _{A_{q}}<ct^{-3/2}\text{, \ \ for any \ \ }%
R>0.
\end{equation*}%
Letting $R\rightarrow \infty $, we get that $s_{t}\in A_{q}$ and that 
\begin{equation*}
\left\Vert s_{t}\right\Vert _{A_{q}}<ct^{-3/2}\text{, \ for \ }\left\vert
t\right\vert >1.
\end{equation*}

The proof in the case $G$ complex is similar and then omitted.
\end{proof}

\subsection{\label{proof}Proof of the results}

Once we have established the necessary ingredients for the proof of
dispersive estimates of the Schr\"{o}dinger operator $\widehat{S}_{t}$ on $%
M\in \left( S\right) $, mainly the norm estimates of the kernel $\widehat{s}%
_{t}(x,y)$ obtained in Proposition \ref{kest2} and the estimates of the
operator norms of $\widehat{S}_{t}$ obtaind in Lemmata \ref{diest1} and \ref%
{diest12a}, we give the proofs of Theorems \ref{diest13} and \ref{diest140}
and present their applications to the study of well-posedness and scattering
for NLS equations.

Having obtained the proofs of the ingredients mentionned above, the proofs
become standard and they are similar to the proofs of the corresponding
results in \cite{ANPI,ANPIVA}. Consequently we shall be brief and present
only their main lines or simply we will omit them.

\begin{proof}[Proof of Theorem \protect\ref{diest13}]
Assume that $G$ is complex. Then, for $|t|<1$ and $2<q,\tilde{q}\leq \infty $%
, using the norm estimates of the kernel $\widehat{s}_{t}(x,y)$ we proved in
Lemma \ref{diest1}, we get that 
\begin{eqnarray*}
\Vert \widehat{S}_{t}\Vert _{L^{1}(M)\rightarrow L^{q}\left( M\right) }
&\leq &c|t|^{-n/2}, \\
\Vert \widehat{S}_{t}\Vert _{L^{\tilde{q}^{\prime }}(M)\rightarrow L^{\infty
}\left( M\right) } &\leq &c|t|^{-n/2}.
\end{eqnarray*}

Also, by the spectral theorem, we have that 
\begin{equation*}
\Vert \widehat{S}_{t}\Vert _{L^{2}(M)\rightarrow L^{2}\left( M\right) }=1.
\end{equation*}

By interpolation we get that there exists a constant $c>0$ such that 
\begin{equation*}
\Vert \widehat{S}_{t}\Vert _{L^{\tilde{q}^{\prime }}(M)\rightarrow
L^{q}(M)}\leq c|t|^{-n\max {\{\frac{1}{2}-\frac{1}{q},\frac{1}{2}-\frac{1}{%
\tilde{q}}\}}},\text{ if }|t|<1.
\end{equation*}

Similarly, for $|t|\geq 1$, by Lemmata \ref{diest1}, we have that 
\begin{eqnarray*}
\Vert \widehat{S}_{t}\Vert _{L^{1}(M)\rightarrow L^{q}(M)} &\leq
&c|t|^{-n/2}, \\
\Vert \widehat{S}_{t}\Vert _{L^{\tilde{q}^{\prime }}(M)\rightarrow L^{\infty
}(M)} &\leq &c|t|^{-n/2}.
\end{eqnarray*}

Further, using Kunze and Stein, we proved in Lemma \ref{diest12a} that 
\begin{equation*}
\Vert \widehat{S}_{t}\Vert _{L^{q^{\prime }}(M)\rightarrow L^{q}(M)}\leq
c\Vert s_{t}\Vert _{A_{q}}\leq c|t|^{-n/2}.
\end{equation*}

By interpolation it follows that for $|t|\geq 1$, 
\begin{equation*}
\Vert \widehat{S}_{t}\Vert _{L^{\tilde{q}^{\prime }}(M)\rightarrow
L^{q}(M)}\leq c|t|^{-n/2}.
\end{equation*}

The rank one case is similar and then omitted.
\end{proof}

\bigskip

\begin{proof}[Proof of of Theorem \protect\ref{diest140}]
As it is allready mentioned in the Introduction, to prove the Strichartz
estimates 
\begin{equation}
\Vert u\Vert _{L_{t}^{p}L_{x}^{q}}\leq c\left\{ \Vert f\Vert
_{L_{x}^{2}}+\Vert F\Vert _{L_{t}^{\tilde{p}^{\prime }}L_{x}^{\tilde{q}%
^{\prime }}}\right\} ,  \label{stri00}
\end{equation}%
for the solutions $u\left( t,x\right) $ of the Cauchy problem 
\begin{equation}
\left\{ 
\begin{array}{l}
i\partial _{t}u(t,x)+\Delta u(t,x)=F(t,x),\text{ \ }t\in \mathbb{R}\text{, \ 
}x\in M, \\ 
u(0,x)=f(x),%
\end{array}%
\right.  \label{lcp1}
\end{equation}%
we shall combine the dispersive estimates of the operator $e^{it\Delta }$
obtained in Theorem \ref{diest13} with the classical $TT^{\ast }$ method.
This method consists in proving $L_{t}^{p^{\prime }}L_{x}^{q^{\prime
}}\rightarrow L_{t}^{p}L_{x}^{q}$ boundedness of the operator 
\begin{equation*}
TT^{\ast }{F}(t,x)=\int_{-\infty }^{+\infty }{\widehat{S}}_{t-s}{F}(s,x)ds
\end{equation*}%
and of its truncated version 
\begin{equation*}
\widetilde{TT}^{\ast }{F}(t,x)=\int_{0}^{t}{\widehat{S}}_{t-s}{F}(s,x)ds,
\end{equation*}%
for all admissible indices $\left( p,q\right) $. For that we shall make also
use of the fact that the solutions of (\ref{lcp1}) are given by Duhamel's
formula: 
\begin{equation}
u(t,x)=e^{it\Delta }f\left( x\right) -i\int_{0}^{t}e^{i\left( t-s\right)
\Delta }F\left( s,x\right) ds.  \label{duha}
\end{equation}

Assume that the pairs $(p,q)$ satisfy 
\begin{equation*}
\tfrac{1}{q}\in \left( \left( \tfrac{{1}}{2}-\tfrac{1}{n}\right) ,\tfrac{1}{2%
}\right) \text{ and }\tfrac{1}{p}\in \left( \left( \tfrac{{1}}{2}-\tfrac{1}{q%
}\right) \tfrac{n}{2},\tfrac{1}{2}\right) .
\end{equation*}%
From (\ref{duha}) it follows that to finish the proof of the theorem, it is
enough to show that 
\begin{equation}
\left\Vert \int_{-\infty }^{+\infty }{\widehat{S}}_{t-s}{F(s)}ds\right\Vert
_{L_{t}^{p}L_{x}^{q}}\leq c\Vert F\Vert _{L_{t}^{\tilde{p}^{\prime }}L_{x}^{%
\tilde{q}^{\prime }}}  \label{stri1}
\end{equation}%
and 
\begin{equation}
\left\Vert {\widehat{S}}_{t}{f(s)}\right\Vert _{L_{t}^{p}L_{x}^{q}}\leq
c\Vert f\Vert _{L_{x}^{2}}.  \label{stri1.5}
\end{equation}%
We give only the proof of (\ref{stri1}). The proof of (\ref{stri1.5}) is
similar. We have that

\begin{eqnarray*}
\left\Vert \int_{-\infty }^{+\infty }{\widehat{S}}_{t-s}{F(s)}ds\right\Vert
_{L_{x}^{q}} &\leq &\int_{-\infty }^{+\infty }\left\Vert {\widehat{S}}_{t-s}{%
F(s)}\right\Vert _{L_{x}^{q}}ds \\
&\leq &\int_{-\infty }^{+\infty }\left\Vert {\widehat{S}}_{t-s}\right\Vert
_{L_{x}^{q^{\prime }}\rightarrow L_{x}^{q}}\left\Vert F\left( s\right)
\right\Vert _{L_{x}^{q^{\prime }}}ds.
\end{eqnarray*}%
Assume that $M$ has rank one. Then from Theorem \ref{diest13} we get that 
\begin{eqnarray*}
\left\Vert \int_{-\infty }^{+\infty }{\widehat{S}}_{t-s}{F(s)}ds\right\Vert
_{L_{t}^{p}L_{x}^{q}} &\leq &\left\Vert \int_{|t-s|\geq 1}|t-s|^{-3/2}\Vert {%
F}(s)\Vert _{L_{x}^{q^{\prime }}}ds\right\Vert _{L_{t}^{p}} \\
&&+\left\Vert \int_{|t-s|\leq 1}|t-s|^{-\left( 1/2-1/q\right) n}\Vert {F}%
(s)\Vert _{L_{x}^{q^{\prime }}}ds\right\Vert _{L_{t}^{p}} \\
&:&=I_{1}+I_{2}.
\end{eqnarray*}

To estimate $I_{1}$ and $I_{2}$ we consider the operators 
\begin{equation*}
T_{1}(f)(t)=\int_{|t-s|\geq 1}|t-s|^{-3/2}f(s)ds
\end{equation*}%
and 
\begin{equation*}
T_{2}(f)(t)=\int_{|t-s|\leq 1}|t-s|^{-\left( 1/2-1/q\right) n}f(s)ds.
\end{equation*}%
Note the kernel $k_{1}(u)=|u|^{-3/2}\chi _{\{|u|\geq 1\}}$ of $T_{1}$ is
bounded on $L^{1}$. So, $T_{1}$ is bounded from $L^{p^{\prime }}\ $to $L^{p}$
for every $p\in \lbrack 2,\infty ]$. Similarly, $k_{2}(u)=|u|^{-\left(
1/2-1/q\right) n}\chi _{\{|u|\leq 1\}}$ is bounded on $L^{1}$ if $\tfrac{1}{q%
}\in \left( \left( \tfrac{{1}}{2}-\tfrac{1}{n}\right) ,\tfrac{1}{2}\right) $%
. This implies that and $T_{2}$ is bounded from $L^{p_{1}}$ to $L^{p_{2}}$
for all $p_{1},p_{2}\in \left( 1,\infty \right) $, such that $0\leq \frac{1}{%
p_{1}}-\frac{1}{p_{2}}\leq 1-\left( \frac{1}{2}-\frac{1}{q}\right) n$. So, 
\begin{equation*}
I_{1},I_{2}\leq c\Vert {F}(s)\Vert _{L_{t}^{\tilde{p}^{\prime }}L_{x}^{%
\tilde{q}^{\prime }}},
\end{equation*}%
for all admissible pairs $\left( p,q\right) $ and $\left( \tilde{p},\tilde{q}%
\right) $ as above. The proof of the case $G$ complex is similar.
\end{proof}

\subsection{\label{appli}Applications of Strichartz estimates}

In this section we apply Strichartz estimates to study well-posedness and
scattering for NLS equations.

Consider the Cauchy problem for the inhomogeneous Schr\"{o}dinger equation
on $M$: 
\begin{equation}
\left\{ 
\begin{array}{l}
i\partial _{t}u(t,x)+\Delta u(t,x)=F(u(t,x)),\text{ \ }t\in \mathbb{R}\text{%
, \ }x\in M, \\ 
u(0,x)=f(x),%
\end{array}%
\right.  \label{lcp2}
\end{equation}%
and assume that $F$ has a power-like nonlinearity of order $\gamma $, i.e. 
\begin{equation*}
|F(u)|\leq c|u|^{\gamma },\,\,\,|F(u)-F(v)|\leq c\left( |u|^{\gamma
-1}+|v|^{\gamma -1}\right) |u-v|\,.
\end{equation*}

Recall that the NLS is globally well-posed in $L^{2}(M)$ if, for any bounded
subset $B$ of $L^{2}(M),$ there exists a Banach space $Y$ continuously
embedded into $C(\mathbb{R};L^{2}(M))$, such that for any $f\in B$, the NLS
has a unique solution $u\in Y$ with $u(0,x)=f(x)$ and the map $%
T:B\rightarrow Y,$ $T(f)=u$ is continuous. Here, as in \cite{ANPI}, we take 
\begin{equation*}
Y=Y_{\gamma }=C(\mathbb{R};L^{2}(M))\cap {L^{\gamma +1}(\mathbb{R};L^{\gamma
+1}(M))},
\end{equation*}%
which is a Banach space for the norm 
\begin{equation*}
\Vert u\Vert _{Y_{\gamma }}=\Vert u\Vert _{L_{t}^{\infty }L_{x}^{2}}+\Vert
u\Vert _{L_{t}^{\gamma +1}L_{x}^{\gamma +1}},
\end{equation*}%
and $B=B\left( 0,\varepsilon \right) \subset L^{2}(M)$.

We have the following result.

\begin{theorem}
\label{wepo}Assume that $M\in \left( S_{0}\right) $ and that $F$ has a
power-like nonlineariry of order $\gamma $. If $\gamma \in (1,1+\frac{4}{n}]$%
, then the NLS (\ref{lcp2}) is globally well-posed for small $L^{2}$ data.
\end{theorem}

Also, one can apply Strichartz's estimates in order to show scattering for
the NLS in the case of power-like nonlinearity of order $\gamma $ and for
small $L^{2}$ data.

\begin{theorem}
\label{scat}Assume that $M\in \left( S_{0}\right) $. Consider the Cauchy
problem (\ref{lcp2}) and assume that $F$ has a power-like nonlinearity of
order $\gamma \in (1,1+\frac{4}{n}]$. Then, for every global solution $u$
corresponding to small $L^{2}$ data, there exists $u_{\pm }\in L^{2}(M)$
such that 
\begin{equation*}
\Vert u(t)-e^{it\Delta }u_{\pm }\Vert _{L^{2}(M)}\rightarrow 0,\text{ as }%
t\rightarrow \pm \infty .
\end{equation*}
\end{theorem}

The proof of Theorems \ \ref{wepo} and \ref{scat} are similar to the proofs
of the corresponding results in \cite{ANPI} and then omitted.

\end{document}